\documentclass[psamsfonts]{amsart}

\usepackage{amssymb,amsmath}
\usepackage{amsthm}
\usepackage{esint}
\usepackage{xcolor}
\usepackage[colorlinks,linkcolor=blue,citecolor=blue,pagebackref,hypertexnames=false, breaklinks]{hyperref}

\newlength{\bibitemsep}
\newlength{\bibparskip}
\let\oldthebibliography\thebibliography
\renewcommand\thebibliography[1]{%
  \oldthebibliography{#1}%
  \setlength{\parskip}{\bibitemsep}%
  \setlength{\itemsep}{\bibparskip}%
}

\newcommand{\brak}[1]{\left(#1\right)}    

\newcommand{\bE}{\mathbb E}
\newcommand{\bG}{\mathbb G}

\newcommand{\bH}{\mathbb H}

\newcommand{\R}{\mathbb R}

\newcommand{\Scal}{\mathcal S}

\newcommand{\g}{\mathfrak g}

\newcommand{\SU}{ \mathbb{ SU}}

\newcommand{\SL}{ \mathbb{ SL}}

\theoremstyle{plain}
\newtheorem{thm}{Theorem}[section]

\newtheorem{lem}[thm]{Lemma}
\newtheorem{cor}[thm]{Corollary}
\newtheorem{prop}[thm]{Proposition}

\theoremstyle{definition}

\newtheorem{remark}[thm]{Remark}
\newcommand{\bremark}{\begin{remark} \em}
\newcommand{\eremark}{\end{remark} }
\usepackage{color}

\hbadness=\maxdimen

\begin{document}
\title{A note on second order Riesz transforms in 3-dimensional Lie groups}
\author{Fabrice Baudoin} 
\address{Department of Mathematics, University of Connecticut, Storrs, CT 06269}
\email{fabrice.baudoin@uconn.edu}
\author{Li Chen}
\address{Department of Mathematics, University of Connecticut, Storrs, CT 06269}
\email{li.4.chen@uconn.edu}
\thanks{F.B. was partly supported by the NSF grant DMS~1901315.}
\date{\today}

\begin{abstract} 
We prove explicit $L^p$ bounds for second order Riesz transforms of the sub-Laplacian in the Lie groups $\mathbb H$, $\mathbb{SU}(2)$ and $\mathbb{SL}(2)$.
\end{abstract}

\maketitle

\section{Introduction}
The classical sharp martingale inequalities, dated back to the work of Burkholder, have many applications in the study of  basic singular integrals on Euclidean spaces. As typical examples, they lead to sharp $L^p$ bounds for Riesz transforms and sharp or dimension-free $L^p$ bounds for second order Riesz transforms. We refer the reader to, for instance, \cite{BB13,BW, GV} and the overview paper \cite{B10}. The purpose of this note is to study explicit $L^p$ bounds for certain second order Riesz transforms, in particular Beurling-Ahlfors type operators, on three dimensional subelliptic model spaces including the Heisenberg group $\bH$,  $\SU(2)$ and $\SL(2)$.

Second order Riesz transforms appear naturally in the study of PDEs (see for instance \cite{GT}) and have been extensively studied in literatures. They are mostly interpreted as iterations of Riesz transforms and their adjoints, or second derivatives of the fundamental solution operator for the Laplacian: $\partial_i \partial_j(-\Delta)^{-1}$. On Euclidean spaces, second order Riesz transforms are well understood as Calder\'on-Zygmund singular integrals and have bounded $L^p$ norm for $1<p<\infty$. A particular interesting example is the classical Beurling-Ahlfors operator on the complex plane defined by 
\[
Bf(z)=\mathrm{p.v.}\frac{1}{\pi}\int_{\mathbb C} \frac{f(w)}{(z-w)^2} dw.
\]
Equivalently,
\[
B=R_1^2-R_2^2-2iR_1R_2,
\]
 where $R_i=\frac{\partial}{\partial x_i}(-\Delta)^{-1/2}$ are the Riesz transforms on $\R^2$. One can also write $B=\partial^2(-\Delta)^{-1}$,
 where $\partial$ is the Cauchy-Riemann operator
 $$\partial=\frac{\partial}{\partial x_1}-i\frac{\partial}{\partial x_2}.$$ 
 The classical  Beurling-Ahlfors operator and its generalizations play an important role in quasiconformal mappings. Sharp or dimension-free $L^p$ bounds  of second order Riesz transforms have been studied  from both deterministic methods and martingale transform techniques using either Poisson or heat extensions, see for instance \cite{BM03,BL97,BW,VN}.  However, the sharp $L^p$ norm of Beurling-Ahlfors operator is a long existing open problem.  

On Heisenberg groups $\bH^n$, second order Riesz transforms are also singular integral operators (see \cite[Theorem 3]{Folland}), which follow from the fundamental solution of the sub-Laplacian $L$ (see for instance \cite{Folland, Stein}).
Dimension-free  $L^p$ bounds for Riesz transforms have been obtained in \cite{CMZ}, see also  \cite{ST12} for the reduced case. However explicit $L^p$ bounds of Riesz transforms and their higher order analogues are not known. We would like to mention that  in \cite{BBC}  the authors together with R. Ba\~nuelos gave explicit $L^p$ bounds for generalized Riesz transforms which are commutator of complex gradients and the square root of non-negative sub-Laplacian. 
In other settings, dimension-free  $L^p$ bounds for second order Riesz transforms have also been studied via  deterministic or probabilistic approaches, for instance,  on $k$-forms on  complete Riemannian manifolds under curvature assumptions \cite{Li11, Li14}, and on discrete groups \cite{ADS, DOP18, DP14}. In this note, we are interested in extensions of Beurling-Ahlfors operator and other second order Riesz transforms on the Lie groups $\bH$,  $\SU(2)$ and $\SL(2)$.

Following \cite{BBBQ, BG17}, given \(\rho \in \R\), let $\bG(\rho)$ be a three dimensional Lie group with Lie algebra $\g$ and we assume that there is a basis $\{X,Y,Z\}$ of $\g$ such that 
\[
[X,Y]=Z, \quad [X,Z]=-\rho Y,\quad [Y,Z]=\rho X.
\]
The cases $\rho=0$, $\rho=1$ and $\rho=-1$ are corresponding, respectively, to Heisenberg group $\bH$, $\SU(2)$ and $\SL(2)$. Consider then on  $\bG(\rho)$ the subelliptic, left-invariant, sub-Laplacian
\[
L=X^2+Y^2,
\]
our main result is the following:

\begin{thm}\label{thm:main}
For any real-valued numbers $a,b,c$ and $\alpha\ge 0$, consider on $\bG(\rho)$ the operator 
\begin{multline}\label{eq:general}
S_{a,b,c}^{\alpha}f=a(-L+\alpha)^{-1}Zf+b\left( X(-L-\rho+\alpha)^{-1}X- Y(-L-\rho+\alpha)^{-1}Y\right)f\\+c\left( X(-L-\rho+\alpha)^{-1}Y+ Y(-L-\rho+\alpha)^{-1}X\right)f.
\end{multline}
Then we have
\[
\|S_{a,b,c}^{\alpha }f\|_{p}\le (|a|+|b|+|c|)(p^*-1)\|f\|_p, 
\]
where $p^*=\max\{p, \frac p{p-1}\}$.
\end{thm}


To prove this result, we write $S_{a,b,c}^{\alpha }$ in terms of the heat semigroup $P_t=e^{tL}$ and then use martingale techniques similar to   \cite[Theorem 1.1]{BB13}. This note is organized as follows. In Section 2, we study explicit $L^p$ bounds of the Beurling-Ahlfors type operator 
$\int_0^{\infty}P_tWWP_t fdt $
in specific settings including Heisenberg group $\mathbb H$ (more generally $\mathbb H^n$), $\SU(2)$ and $\SL(2)$. In Section 3, we give the proof of our main result Theorem \ref{thm:main}.

\section{Specific settings}

\subsection{On Heisenberg groups}

The three dimensional simply connected Heisenberg group is the space $\bG(\rho)$ with $\rho=0$. 
For the sake of generality, we consider the Heisenberg group $\bH^n=\R^{2n+1}$ with arbitrary $n\in \mathbb N$.
A basis of left-invariant vector fields is
\[
X_j=\frac{\partial}{\partial x_j}-\frac{y_j}{2}\frac{\partial}{\partial z},\quad Y_j=\frac{\partial}{\partial y_j}+\frac{x_j}{2} \frac{\partial}{\partial z}, \quad Z=\frac{\partial}{\partial z},
\]
and hence the sub-Laplacian 
\[
L=\sum_{j=1}^n \left(X_j^2+Y_j^2\right).
\]
Similarly as in $\bH$, we have for any $1\le j\le n$
\[
[X_j,Y_j]=Z, \quad [X_j,Z]= [Y_j,Z]=0.
\]
Let $W_j=X_j-iY_j$ be the complex gradient, then
\begin{equation}\label{eq:WL}
W_jL=(L+2iZ)W_j.
\end{equation}
In other words, we have $[W_j,L]=2iZW_j$. Note also that $[W_j,Z]=[L,Z]=0$.

Using the cylindric coordinates $(r,\theta,z)$, the radial sub-Laplacian is then
\[
L=\frac{\partial^2}{\partial r^2}+\frac{2n-1}r\frac{\partial}{\partial r}+\frac{r^2}{4}\frac{\partial^2}{\partial z^2}.
\]
In particular, one can write on $\mathbb H$ the left-invariant complex gradient as
\[
W=X-iY=e^{-i\theta} \frac{\partial}{\partial r}-\frac{ie^{-i\theta}}{r} \frac{\partial}{\partial r}-\frac{ie^{-i\theta}}{2} r\frac{\partial}{\partial z},
\]
and  the right-invariant complex gradient as
\[
\hat W=\hat X-i\hat Y=e^{-i\theta} \frac{\partial}{\partial r}-\frac{ie^{-i\theta}}{r} \frac{\partial}{\partial r}+\frac{ie^{-i\theta}}{2} r\frac{\partial}{\partial z}.
\]

We are concerned with the second order Riesz transforms   $W_j(-L)^{-1}W_k$, $1\le j,k\le n$. Denote by $\Scal(\bH^n)$ the Schwartz space of smooth rapidly decreasing functions on the Heisenberg group. 
\begin{prop}\label{prop:mainH}
Let $f\in \Scal(\bH^n)$. Then  $W_j(-L)^{-1} W_kf$, $1\le j,k\le n$, are bounded on $L^p(\bH^n)$ with 
\begin{align*}
\|W_j(-L)^{-1} W_k f\|_p&\le \sqrt 2\,(p^*-1)\|f\|_p.
\end{align*}
\end{prop}

\begin{remark} 
We obtain on $\bH$ an explicit bound for the Beurling-Ahlfors operator $B_{\bH}=W(-L)^{-1}W$. As singular integral  operator, $B_{\bH}$ has the form
\[
B_{\bH}f(0)=\frac1{\pi}\int_{\mathbb H} (x-iy)^2(|x+iy|^4+16z^2)^{-3/2} f(\mathbf x) d\mathbf x.
\]
Indeed, at the origin one has
\[
W(-L)^{-1}Wf(0)=(-L)^{-1}\hat WWf(0)=\int_{\bH} G(\mathbf x) \hat WWf(\mathbf x) d\mathbf x=\int_{\bH} W\hat WG(\mathbf x)f(\mathbf x) d\mathbf x, 
\]
where $G(\mathbf x)$ is the Green function on $\bH$ (see for instance \cite[Theorem 2]{Folland}):
\[
G(\mathbf x)=G(x,y,z)=\frac1{4\pi \sqrt{\brak{x^2+y^2}^2+16z^2}}.
\]
Direct computation from the cylinder coordinates yields
\[
W\hat W G(r, \theta, z)=\frac{r^2e^{-2i\theta}}{\pi (r^4+16z^2)^{3/2}} . 
\]
On the other hand, let $\Pi: \mathbb H\to \mathbb C$ be the projection operator, then 
\begin{align*}
B_{\bH}(f\circ \Pi)(0) 
=\frac1{\pi} \int_{\mathbb C}\frac{1}{ w^2}f(w) dw.
\end{align*}
This is exactly the classical Beurling-Ahlfors operator.
\end{remark}

Our proof for Proposition \ref{prop:mainH} is through martingale transform techniques. In order to obtain probabilistic representations 
$W_j(-L)^{-1}W_k$, $1\le j,k\le n$, we first try to rewrite these operators in form of  the Littlewood-Paley type.  
Formal computation from \eqref{eq:WL} leads to 
\[
W_jP_t=e^{-2itZ} P_tW_j
\]
and hence $P_tW_jW_kP_t=W_jP_t e^{2itZ}e^{-2itZ} P_tW_k=W_jP_{2t}W_k$. 
Thus one may have
\begin{equation}\label{eq:WLW-Heisenberg}
W_j(-L)^{-1}W_kf=2\int_0^{\infty} P_t W_j W_kP_t f dt,\quad \forall f \in \Scal(\bH^n).
\end{equation}
The above computations are only formal since the semigroup associated $L-2iZ$ and $L+2iZ$ are not globally well-defined, see \cite[Section 5.2]{BBBC} for more details. However, we can rigorously show the following lemma.

\begin{lem}\label{lem:RieszRepr}
Let $ f \in \Scal(\bH^n)$, then \eqref{eq:WLW-Heisenberg} holds.
\end{lem}

\begin{proof}
If $f(\mathbf x)=f(x,y, z)=e^{i\lambda z} g(x,y)$ for some $\lambda \in \R$ and some function $g$, we have $Zf=i\lambda f$. It follows from \eqref{eq:WL} that  for any $1\le j\le n$,
\[
W_jLf=(L+2\lambda )W_jf.
\]
We deduce then 
\begin{equation}\label{eq:WPt}
W_jP_t f(0)=e^{2\lambda t}(P_tW_jf)(0).
\end{equation}
 Denote $f_j(x,y, z)=W_jf(x,y, z)$. Observe that $ZW_jf=W_jZf=i\lambda W_jf$, i.e., $Z  f_j=i\lambda f_j$. Therefore
\[
W_jW_kP_t f(0)=e^{2\lambda t}W_j(P_tW_kf)(0)=e^{2\lambda t}W_jP_tf_k(0)=e^{4\lambda t}(P_tW_j f_k)(0),
\]
and thus applying \eqref{eq:WPt} again,
\[
P_tW_jW_kP_t f(0)=e^{4\lambda t}(P_{2t}W_j f_k)(0)=W_j P_{2t}  f_k(0) =W_j P_{2t} f_k(0).
\]
Now plugging this into the left hand side of \eqref{eq:WLW-Heisenberg} leads to 
\[
2 \int_0^{\infty} W_j P_{2t} W_k f dt=2W_j \int_0^{\infty}  P_{2t} dt W_k f =W_j(-L)^{-1} W_kf.
\]

For general $f$, we can conclude by using the Fourier transform of $f$ with respect to the variable $z$.
\end{proof}

Thanks to the identity 
\eqref{eq:WLW-Heisenberg}, we are now able to give a probabilistic proof for the sharp or dimension-free explicit $L^p$ bound of $W_j(-L)^{-1}W_k$, $1\le j,k\le n$, following \cite{BB13}. We introduce some notations here. Let $(Y_t)_{t\ge 0}$ be the diffusion associated with the sub-Laplacian $L$ started with a distribution $\mu$ and let $(B_t)_{t\ge 0}$ be the Brownian motion on $\R^{2n}$ associated with the Laplace operator $\Delta$. Denote $\nabla=(X_1,\cdots, X_n,Y_1, \cdots, Y_n)$.

\begin{proof}[\bf Proof of Proposition \ref{prop:mainH}]
Notice that from \eqref{eq:WLW-Heisenberg} one writes
\[
W_j(-L)^{-1}W_kf=2\int_0^{\infty}P_t(X_jX_k-Y_jY_k-i(X_jY_k+X_kY_j))P_t fdt,
\]
The rest of the proof is similar to \cite[Theorem 1.1]{BB13} and we write it here for the sake of completeness. Denote 
\[
S_{jk}^Tf:=\bE\brak{\int_0^T A_{jk}\nabla P_{T-t}f(Y_t) \circ dB_t \mid Y_T=x}.
\]
We claim that 
\[
W_j(-L)^{-1}W_k f(x)=\lim_{T\to \infty}S_{jk}^Tf,
\]
where  $A_{jk}$ is a $2n\times 2n$ matrix with entries $a_{jk}=-1$, $a_{j(n+k)}=a_{(n+j)k}=i$, $a_{(n+j)(n+k)}=1$, and otherwise $0$. 

For any $T>0$, it is easy to see that $(P_{T-t}f(Y_t))_{0\le t\le T}$ is a martingale. 
Then by It\^o's formula and the It\^o isometry, we have for any $g\in \Scal(\bH^n)$ 
\begin{align*}
\int_{\bH^n} g(x)S_{jk}^Tf(x) d\mu(x)&=
\bE \brak{g(Y_T) \int_0^T A_{jk} \nabla P_{T-t}f(Y_t) \circ dB_t}
\\&=
2 \bE \brak{\int_0^T A_{jk} \nabla P_{T-t}f(Y_t) \cdot \nabla P_{T-t}g(Y_t)dt}
\\ &=
2\int_0^{T}\int_{\bH^n}A_{jk} \nabla P_{t}f(x) \cdot \nabla P_{t}g(x)d\mu(x) dt
\\ &=
2\int_{\bH^n}\int_0^{T} P_tW_jW_kP_t f(x) dt g(x)d\mu(x).
\end{align*}
Thus we deduce that 
\[
\lim_{T\to \infty} \int_0^T \int_{\bH^n}S_{jk}^Tf(x)g(x)d\mu(x)=\int_{\bH^n} W_j(-L)^{-1}W_kf(x) g(x)d\mu(x).
\]

Observe that the matrix norm of $A_{jk}$ is $\sqrt 2$, thus by the $L^p$ bound of  martingale transform (see \cite{BW}), 
\[
\|W_j(-L)^{-1}W_kf\|_p\le \sqrt2(p^*-1)\|f\|_p.
\]

\end{proof}


\subsection{On $\SU(2)$}
Consider the Lie group $\SU(2)$, which is $\bG(\rho)$ with $\rho=1$. Denote by $X, Y , Z$ the left invariant vector fields on $\SU(2)$ corresponding to the Pauli matrices
\[
X= \frac12\begin{pmatrix}
    0 & 1 \\
    -1 & 0 
\end{pmatrix},
\quad 
Y= \frac12\begin{pmatrix}
    0 & i \\
    i &  0 
\end{pmatrix},
\quad
Z= \frac12\begin{pmatrix}
    i & 0 \\
    0 & -i 
\end{pmatrix},
\]
for which the commutation relations hold
\[
[X,Y]=Z, \quad [X,Z]=-Y, \quad [Y,Z]=X.
\]

We shall be interested in the sub-Laplacian
\[
L=X^2+Y^2.
\]
Note that $[L,Z]=0$.
Consider also the complex gradient $W=X-iY$. The Lie algebra structure gives us
\begin{equation}\label{eq:WL(SU)}
WL=(L+2iZ+1)W, \quad WZ=(Z-i)W.
\end{equation}

We use the cylindric coordinates introduced in \cite{CS01} 
\[
(r,\theta,z) \to \exp(r\cos \theta X+r\sin \theta Y) \exp(zZ)=\begin{pmatrix}
   \cos\frac r2 e^{i\frac z2} & \sin \frac r2 e^{i(\theta-\frac z2)} \\
    -\sin\frac r2 e^{-i(\theta-\frac z2)} & \cos \frac r2 e^{-i\frac z2}
\end{pmatrix}
\]
with 
\[
0\le r\le \pi,\quad \theta\in [0,2\pi], \quad z\in [-2\pi,2\pi].
\]
Then the vector fields $X$, $Y$ and $Z$ read as (see \cite{BB09})
\begin{align*}
X&= \cos(- \theta+z) \frac{\partial}{\partial r}+\sin(- \theta+z)\brak{\tan\frac r2  \frac{\partial}{\partial z}+\frac12\brak{\tan\frac r2+\frac1{\tan\frac r2}}\frac{\partial}{\partial {\theta}}}
\\
Y&=- \sin(- \theta+z) \frac{\partial}{\partial r}+\cos(- \theta+z)\brak{\tan\frac r2  \frac{\partial}{\partial z}+\frac12\brak{\tan \frac r2+\frac1{\tan \frac r2}}\frac{\partial}{\partial {\theta}}}
\\ 
Z&=\frac{\partial}{\partial z}.
\end{align*}
The complex gradient $W=X-iY$ becomes 
\begin{equation}\label{eq:complexSU}
W= e^{i(-\theta+z)} \frac{\partial}{\partial r}- i e^{i(-\theta+z)} \brak{\tan \frac r2  \frac{\partial}{\partial z}+\frac12\brak{\tan\frac r2+\frac1{\tan\frac r2}}\frac{\partial}{\partial {\theta}}}.
\end{equation}

Our main result on $\SU(2)$ is the following.
\begin{prop}\label{prop:mainSU(2)}
Let $\alpha\ge 0$. Then the second order Riesz transforms $W(-L-1+\alpha)^{-1} W$ is bounded on $L^p$ with 
\[
\|W(-L-1+\alpha)^{-1} W f\|_p\le \sqrt 2\,(p^*-1)\|f\|_p.
\]
\end{prop}

The proof is similarly as that of Proposition \ref{prop:mainH}. 
We observe that formally the relations in \eqref{eq:WL(SU)} lead to 
\[
WP_t=e^{2itZ+t} P_tW
\]
and 
\[
P_tWWP_t=e^{-2itZ-t} WP_t e^{2itZ+t}  P_tW=e^{2t}WP_{2t}W.
\] 
Thus
\begin{equation}\label{eq:L4}
W(-L-1)^{-1}W f=2\int_0^{\infty} P_t W WP_t fdt.
\end{equation}
Indeed, we have
\begin{lem} \label{lem:SU}
Let  $f\in \mathcal S(\SU(2))$, then \eqref{eq:L4} holds.
\end{lem}

\begin{proof}
Similarly as on Heisenberg groups, it suffices to consider $f(r,\theta, z)=e^{i\lambda z} g(r,\theta)$, for some $\lambda \in \R$ and some function $g$. We have $Zf=i\lambda f$. It follows from \eqref{eq:WL(SU)} that
\[
ZWf=W(Z+i)f=i(\lambda+1) W f.
\]
 and
\[
WLf=(L-2\lambda-1 )Wf.
\]
We deduce then 
\begin{equation}\label{eq:WPt-SU1}
WP_t f=e^{-(2\lambda +1)t}P_tWf.
\end{equation}

Next applying \eqref{eq:WL(SU)} again, one gets
\[
ZWWf=W(Z+i)Wf=WW(Z+2i)f=i(\lambda+2) WW f.
\]
Therefore
\[
WLWf=(L+2iZ+1 )WWf=(L-2\lambda-3) WWf,
\]
and
\begin{equation}\label{eq:WPt-SU2}
WP_t W f=e^{-(2\lambda+3)t}P_tWW f.
\end{equation}

Now combining  \eqref{eq:WPt-SU1} and \eqref{eq:WPt-SU2}, we  obtain
\[
WWP_t f=e^{-(2\lambda+1) t}W(P_tWf)=e^{-(2\lambda+3) t} e^{-(2\lambda+1) t}P_tW W f=e^{-(4\lambda+4) t}P_tWW f.
\]
Using \eqref{eq:WPt-SU2} again, then
\[
P_tWWP_t f=e^{-(4\lambda+4)t}P_{2t}W W f=e^{-(4\lambda+4) t}e^{2(2\lambda+3) t}W P_{2t} W f =e^{2 t} W P_{2t} W f.
\]
This leads to 
\begin{align*}
\int_0^{\infty} P_t W WP_t  dt f&=e^{2 t}  \int_0^{\infty} W P_{2t} W f dt=e^{2 t} W \int_0^{\infty}  P_{2t} dt W f \\ &=\frac12W(-L-1)^{-1} Wf.
\end{align*}
\end{proof}

\begin{remark}
\

\begin{enumerate}
\item The operator $W(-L-1)^{-1}W$ is always well-defined thanks to the spectral decomposition of the sub-Laplacian. Indeed,  the eigenvalues of $L$ are 
\[
-\lambda_{n,k}=-k(k+|n|+1)-\frac{|n|}2, 
\]
where $n \in \mathbb Z, k\in \mathbb N \cup \{0\}$ (see \cite{BB09} for more details). Denote the corresponding eigenfunctions by $\varphi_{n,k}$. When $k=0$ and $n\ge 0$, one has $W \varphi_{n,k}=0$. When $k=0$ and $n=-1$, one has $WW\varphi_{-1,0}=0$ 
and thus
\[
W(-L-1)^{-1}W \varphi_{-1,0}= 0.
\] 
These indicate that the operator $W(-L-1)^{-1}W$ on eigenfunctions corresponding to eigenvalues $-1/2$ vanishes.

\item Let $\alpha\ge 0$. The same proof as above gives
\[
W(-L-1+\alpha)^{-1}W f=2\int_0^{\infty} e^{-2\alpha t} P_t W WP_t dt.
\]
In particular, taking $\alpha=1$, we obtain $W(-L)^{-1}W$ on the right hand side.
\end{enumerate}
\end{remark}

\begin{proof}[Proof of Proposition \ref{prop:mainSU(2)}]
Due to the identification \eqref{eq:L4},  one has (see Proposition \ref{prop:mainH} and also \cite{BB13}) 
\[
W(-L-1+\alpha)^{-1} W f(x)=\lim_{T\to \infty}\bE\brak{\int_0^T A(T-t)\nabla P_{T-t}f(Y_t) \circ dB_t \mid Y_T=x},
\]
where  $A$ is a $2\times 2$ matrix with entries $a_{11}=-e^{-2\alpha t}$,  $a_{12}=a_{21}=e^{-2\alpha t}i$ and $a_{22}=e^{-2\alpha t}$. Observe that the matrix norm is bounded by $\sqrt 2$. This leads to the desired estimate. 
\end{proof}

\subsection{$\SL(2)$}
Consider the Lie group $\SL(2)$, which is $\bG(\rho)$ with $\rho=-1$. Denote by $X, Y , Z$ the left invariant vector fields on $\SL(2)$ corresponding to the basis of $\mathfrak{sl}(2)$:
\[
X= \frac12\begin{pmatrix}
    1 & 0 \\
    0 & -1 
\end{pmatrix},
\quad 
Y= \frac12\begin{pmatrix}
    0 & 1 \\
    1 &  0 
\end{pmatrix},
\quad
Z= \frac12\begin{pmatrix}
    0 & 1 \\
    -1 & 0
\end{pmatrix},
\]
Then the following relations hold
\[
[X,Y]=Z, \quad [X,Z]=Y, \quad [Y,Z]=-X.
\]
Consider the sub-Laplacian $L=X^2+Y^2$ and the complex gradient $W=X-iY$. The Lie algebra structure gives us
\begin{equation}\label{eq:WL(SL)}
WL=(L+2iZ-1)W, \quad WZ=(Z+i)W.
\end{equation}

Similarly as on $\SU(2)$, we can show
\begin{lem}\label{lem:SL}
On $\SL(2)$, one has
\begin{equation*}\label{eq:L-4}
W(-L+1)^{-1}W f=2\int_0^{\infty} P_t W WP_t dt.
\end{equation*}
\end{lem}

\begin{prop}The second order Riesz transforms $W(-L+1)^{-1} W$ is bounded on $L^p$ with 
\[
\|W(-L+1)^{-1} W f\|_p\le \sqrt 2\,(p^*-1)\|f\|_p.
\]
\end{prop}
The proofs are very similar to the case of $\SU(2)$ so we omit the details here.


\section{Proof of Theorem  \ref{thm:main}}

In this section, we will prove our main result. First recall that from Ba\~nuelos and Baudoin \cite{BB13} we have the following result.
\begin{lem}\label{lem:SA} 
Consider  $\mathbb G(\rho)$. Let $A=(a_{ij})_{2\times 2}$ be a real or complex-valued matrix (may or may not depending on $t$).
Then we have 
\[
S_Af=\int_0^{\infty} P_t(a_{11}X^2+a_{12}XY+a_{21}YX+a_{22}Y^2)P_t fdt
\]
and
\[
\|S_A f\|_p \le \frac12(p^*-1)\|A\| \|f\|_p.
\]
In particular, if $A$ is real and orthogonal, then 
\[
\|S_A f\|_p \le \frac12\cot\left(\!\frac{\pi}{2p^*}\!\right)  \|f\|_p.
\]
\end{lem}

On three dimensional model space $\mathbb G(\rho)$, it always holds that $[L,Z]=0$ and  $[X,Y]=Z$. Hence taking $a_{12}=1$, $a_{21}=-1$ and $a_{11}=a_{22}=0$, we  have 
\[
Z (-L)^{-1}=2\int_0^{\infty} P_t(XY-YX)P_t fdt .
\]
More generally, for any $\alpha\ge 0$, 
\[
Z (-L+\alpha)^{-1}=2e^{-2\alpha}\int_0^{\infty} P_t(XY-YX)P_t fdt .
\]
As a consequence of Lemma \ref{lem:SA}, we obtain
\begin{cor}\label{lem:Z}
Consider a three dimensional model space $\mathbb G(\rho)$. Let $1<p<\infty$ and $\alpha\ge 0$. Then we have
\[
\|Z (-L+\alpha)^{-1}f\|_p\le \cot\left(\!\frac{\pi}{2p^*}\!\right) \|f\|_p. 
\]
\end{cor}
The sub-Laplacian does not commute with $X$ and $Y$ on  $\mathbb G(\rho)$. However, the complex gradient $W=X-iY$ (and its conjugate) and the sub-Laplacian are well related, see Section 2. Combining Lemmas \ref{lem:RieszRepr}, \ref{lem:SU} and \ref{lem:SL}, one concludes that for $\alpha\ge 0$, there holds
\begin{equation}\label{eq:Int2ndRiesz}
W(-L-\rho+\alpha)^{-1}=2e^{-2\alpha t}\int_0^{\infty} P_t WWP_t fdt.
\end{equation}
Therefore we have from Lemma \ref{lem:SA} that
\begin{cor} \label{lem:WLW}
Consider a three dimensional model space $\mathbb G(\rho)$. Let $1<p<\infty$ and $\alpha\ge 0$. Then 
\[
\|W(-L-\rho+\alpha)^{-1}W f\|_p\le \sqrt 2(p^*-1) \|f\|_p. 
\]
\end{cor}

Now we are ready to prove our main result.
\begin{proof}[Proof of Theorem \ref{thm:main}]
Comparing the real part and the imaginary part for both sides of \eqref{eq:Int2ndRiesz} yields that for $\alpha\ge 0$,
\begin{equation*}\label{eq:ReIm}
\begin{split}
&\left(X(-L-\rho+\alpha)^{-1}X-Y(-L-\rho+\alpha)^{-1}Y\right)f=2e^{-2\alpha}\int_0^{\infty} P_t (X^2-Y^2)P_t fdt ,
\\&
\left(X(-L-\rho+\alpha)^{-1}Y+Y(-L-\rho+\alpha)^{-1}X\right)f=2e^{-2\alpha}\int_0^{\infty} P_t (XY+YX)P_t fdt .
\end{split}
\end{equation*}
Thus by Corollary \ref{lem:WLW} one has 
\begin{equation}\label{eq:XLY}
\begin{split}
&\|\left(X(-L-\rho+\alpha)^{-1}X-Y(-L-\rho+\alpha)^{-1}Y\right)f\|_p \le \sqrt 2(p^*-1)\|f\|_p,
\\
&\|\left(X(-L-\rho+\alpha)^{-1}Y+Y(-L-\rho+\alpha)^{-1}X\right)f\|_p \le \sqrt 2(p^*-1)\|f\|_p.
\end{split}
\end{equation}
The $L^p$ bound of $S_{a,b,c}^{\alpha }$ then follows from \eqref{eq:XLY} and Corollary \ref{lem:Z}.
\end{proof}

\bibliographystyle{plain}

\end{document}